\documentclass{article}
\usepackage[utf8]{inputenc}
\usepackage[babel]{microtype}
\usepackage[english]{babel}
\usepackage{fullpage}
\usepackage{comment}

\usepackage{amsmath,amssymb,amsthm, mathrsfs, mathtools,bbm,bm}
\usepackage[usenames,dvipsnames]{xcolor}
\usepackage{multirow}
\usepackage{tikz-cd,tikz}
\usetikzlibrary{arrows.meta, positioning}
\usetikzlibrary{decorations.pathreplacing, calligraphy}
\usepackage{enumitem}
\setlist[enumerate,1]{label={(\roman*)}}
\setlist[enumerate,2]{label={(\alph*)}}
\setlist[enumerate,3]{label={(\arabic*)}}
\tikzcdset{scale cd/.style={every label/.append style={scale=#1},
    cells={nodes={scale=#1}}}}
\usepackage[hypertexnames=false]{hyperref}
\hypersetup{
    colorlinks,
    linkcolor={red!60!black},
    citecolor={green!60!black},
    urlcolor={blue!60!black}
}

\newcommand{\dd}{\mathcal{D}}
\newcommand{\lp}{\mathcal{P}}
\newcommand{\SSS}{\mathfrak{S}}
\newcommand{\eul}{\mathcal{A}}

\newcommand{\N}{\mathbb{N}}

\DeclareMathOperator{\wt}{wt}

\DeclareMathOperator{\ch}{ch}
\DeclareMathOperator{\cl}{cl}
\DeclareMathOperator{\rev}{rev}
\DeclareMathOperator{\bl}{bl}

\newtheorem{theorem}{Theorem}[section]
\newtheorem{proposition}[theorem]{Proposition}
\newtheorem{corollary}[theorem]{Corollary}

\theoremstyle{definition}
\newtheorem*{remark}{Remark}
\newtheorem{example}{Example}
\numberwithin{equation}{section}
 
\usepackage{authblk}
\title{Log-concavity of rows of triangular arrays satisfying a certain super-recurrence}
\author[1]{Umesh Shankar\thanks{\tt{204093001@iitb.ac.in, umeshshankar@outlook.com}}} 
\affil[1]{Department of Mathematics, Indian Institute of Technology Bombay, Mumbai 400076, India} 
\date{\today}
\begin{document}
\maketitle
\begin{abstract}
Recurrences of the form 
\begin{equation*}
    T(n,k) = (\alpha n+\beta k +\gamma) \ T(n-1,k) + (\alpha'n+\beta'k+\gamma')\ T(n-1,k-1)+\delta_{n,0}\delta_{k,0}.
\end{equation*}
show up as the recurrence for many well-studied combinatorial sequences such as the Stirling numbers of first and second kind, the Lah numbers, Eulerian numbers etc. Recently, many of these sequences have received generalisations that obey a recurrence of the form
\begin{equation*}
    T(n,k) = (\alpha n+\beta k +\gamma)^l \ T(n-1,k) + (\alpha'n+\beta'k+\gamma')^l\ T(n-1,k-1)+\delta_{n,0}\delta_{k,0}.
\end{equation*}
where $l$ is a positive integer. Many of these generalised sequences also satisfy properties such as unimodality, log-concavity, gamma-nonnegativity, real-rootedness that the original sequences satisfy. In this article, we give sufficient conditions for rows of triangular arrays, arising from the recurrence stated above, to be log-concave. We show that our sufficient condition is satisfied by many of the classical examples, thereby giving a new unified approach to proving their log-concavity. This sufficient condition also confirms a conjecture of Tankosi{\v c} \cite{aleks-gen-lah} about the log-concavity of generalised Lah numbers. 

Our main technique will be to interpret the triangular array $(T(n,k))$ as weighted lattice paths and produce an injection that is increasing in weight.
Finally, we introduce a two-parameter generalisation of the Eulerian numbers analogous to the generalised Stirling and Lah counterparts. We prove that this sequence is palindromic and make some remarks about their gamma-nonnegativity and real-rootedness.
\end{abstract}
\section{Introduction}
Let $\N$ denote the set of non-negative integers and for $n \in \N$, let $[n]$ denote the set $\{ 1,2,\dotsc, n\}$. A finite sequence of real numbers $(a_k)_{k=0}^n$ is said to be log-concave if $a_{k}^2\ge a_{k+1}a_{k-1}$ for all indices $k\in [n-1]$. Log-concavity is a well-studied property of finite sequences and it appears in various areas of mathematics such as combinatorics, probability and algebra. See the surveys of Stanley \cite{stanley-survey-logcon}, Branden \cite{branden-unimodality_log_concavity} for a wealth of results on log-concavity.

A \emph{triangular array} $T(n,k)$ is a collection of real numbers that satisfies a recurrence of the form 
\begin{equation}\label{Eq:triangular}
    T(n,k) = c(n,k)\ T(n-1,k) + d(n,k)\ T(n-1,k-1) + \delta_{n,0}\delta_{k,0}, \qquad \text{for }n, k \geq 0,
\end{equation}
and $T(n,k) = 0$ for $n < 0$, $k < 0$, or $k > n$, for some $c(n,k), d(n,k) \in \mathbb{R}$ and $\delta_{i,j}$ is the Kronecker delta function.
Several interesting enumerative sequences obey a recurrence of this form, including the binomial numbers, $\binom{n}{k}$, the Stirling numbers of the first and second kinds, $s(n,k)$ and $S(n,k)$, the Lah numbers, $L(n,k)$, and the Eulerian numbers, $A(n,k)$.
These sequences and their generating functions are rich in combinatorial structure, and well-studied in the literature, in particular with respect to properties such as log-concavity, unimodality, gamma positivity, and real-rootedness.
Moreover, in all these examples we have that the expressions for $c(n,k)$ and $d(n,k)$ are affine functions in the variables $n$ and $k$.


To prove the log-concavity of the Stirling numbers of both kinds, Kurtz \cite{kurtz-log-concavity} gave the following sufficient condition for triangular arrays.

\begin{theorem}[{\cite[Theorem 2]{kurtz-log-concavity}}]
    Let $T(n,k)=c(n,k)\ T(n-1,k) + d(n,k)\ T(n-1,k-1)$ and 
    \begin{itemize}
        \item $2c(n,k)\ge c(n,k-1)+c(n,k+1),$
        \item $2d(n,k)\ge d(n,k-1)+d(n,k+1),$
    \end{itemize}
     then, the array $T(n,k)$ is log-concave in $k$, for all $n$.
\end{theorem}

A strengthening of the Theorem was obtained by Sagan \cite{sagan-log-concave} wherein the hypothesis of concavity of the coefficients was weakened to log-concavity of the coefficients and concavity of the product. The theorem is as follows.

\begin{theorem}[{\cite[Theorem 1]{sagan-log-concave}}]
    Let $T(n,k)$ be an extended triangular array satisfying $T(n,k)=c(n,k)\ T(n-1,k) + d(n,k)\ T(n-1,k-1)$ for all $n\ge 1$ and $T(n,k),c(n,k),d(n,k)$ are all non-negative integers. Suppose 
    \begin{itemize}
        \item $c(n,k),d(n,k)$ are log-concave in $k$,
        \item $d(n,k-1)c(n,k+1)+d(n,k+1)c(n,k-1)\le 2d(n,k)c(n,k)$,
    \end{itemize}
    then, the $T(n,k)$ are log-concave in $k$.
\end{theorem}

Recently, some of the enumerative sequences listed above have been generalized by introducing more parameters.
Let $r \in \N$.The \index{$r$-Stirling number of the first kind}$r$-Stirling number of the first kind, $s_r(n,k)$ (with $n\ge k\ge r$), is defined to be the number of permutations 
in $\SSS_n$ with $k$ cycles such that $1,2,\dotsc, r$ are in different cycles.
The \index{$r$-Stirling number of the second kind}$r$-Stirling number of the second kind, $S_r(n,k)$ (with $n\ge k\ge r$), is defined to be the number of partitions
 of $[n]$ into $k$ blocks such that $1,2,\dots,r$ are in different blocks. These numbers have been studied in great detail, see \cite{broder-r-stir, stirling-ref-2}.
The closely related\index{$r$-Lah number} $r$-Lah number, $L_r(n,k)$ (with $n\ge k\ge r$), is defined to be the number of partitions of 
$[n]$ into $k$ nonempty ordered subsets such that $1,2,\dots,r$ are in different ordered subsets.
These were studied in \cite{r-lah-ref-1,r-lah-ref-2,lah-ref-3}.
The $r$-Stirling numbers and $r$-Lah numbers are also triangular arrays in the sense of \eqref{Eq:triangular}, with 
coefficients that are affine functions in the variables.

These numbers have been further generalised with an additional parameter as follows (see \cite{lr-stir-ref,aleks-gen-lah}).
Let $l \in \mathbb{N}$.
Define the \index{cycle leader}\emph{cycle leader} of a permutation $\pi \in \SSS_n$, denoted $\cl(\pi)$, to be the set of minimum elements 
of each cycle in the cycle decomposition of $\pi$.
The $(l,r)$-Stirling number of the first kind, $s_r^{(l)}(n,k)$, is defined to be the number of $l$-tuples of permutations
 $(\pi_1,\dotsc,\pi_l)$ such that $[r]\subset \cl(\pi_1)$ and $\cl(\pi_1)=\dotsb=\cl(\pi_l)$. 
Define the\index{block leader} \emph{block leader} of a partition $\Pi$ of $[n]$, denoted $\bl(\Pi)$, to be the set of minimum elements of each
 block of the partition $\Pi$.
The $(l,r)$-Stirling number of the second kind, $S_r^{(l)}(n,k)$, is defined similarly: it is the number of $l$-tuples of 
partitions $(\Pi_1,\dotsc,\Pi_l)$ such that $[r]\subset \bl(\Pi_1)$ and $\bl(\Pi_1)=\dotsb=\bl(\Pi_l)$. 
The $(l,r)$-Lah number is defined in a similar manner to the $(l,r)$-Stirling number of the second kind.
It counts the number of $l$-tuples of partitions $(\Pi_1,\dotsc,\Pi_l)$ into $k$ linearly ordered blocks such 
that $[r]\subset \bl(\Pi_1)$ and $\bl(\Pi_1)=\dotsb=\bl(\Pi_l)$.

The $(l,r)$-Stirling numbers and $(l,r)$-Lah numbers are also triangular arrays; in particular, equation~\eqref{Eq:triangular} for these numbers take the following forms:
\begin{align*}
    s_r^{(l)}(n,k) &= (n-1)^l s_r^{(l)}(n-1,k) + s_r^{(l)}(n-1,k-1)\\
    S_r^{(l)}(n,k) &= k^l S_r^{(l)}(n-1,k) + S_r^{(l)}(n-1,k-1)\\
    L_r^{(l)}(n,k) &= (n+k-1)^l L_r^{(l)}(n-1,k) + L_r^{(l)}(n-1,k-1)
\end{align*}
for $1 \leq r \leq k \leq n$.
While it is easy to show that the $(l,r)$-Stirling numbers of both kinds are log-concave, it was conjectured in \cite{aleks-gen-lah} that the $(l,r)$-Lah numbers are log-concave in $k$ as well. However, the theorems of Kurtz and Sagan are not applicable to these arrays when $l\ge 2$. 

Motivated by the form of the recurrences appearing in the conjecture, we develop a general method to establish log-concavity of many so-called ``super-recurrence'' relations of the form
\begin{equation}\label{Eq:super-recurrence}
    T(n,k)=(\alpha n+\beta k +\gamma)^l\ T(n-1,k) + (\alpha'n+\beta'k+\gamma')^l\ T(n-1,k-1) + \delta_{n,0}\delta_{k,0}, \quad \text{for } n, k \geq 0,
\end{equation}
and $T(n,k) = 0$ for $n < 0$, $k < 0$, or $k > n$, where $\alpha,\alpha',\beta,\beta',\gamma,\gamma' \in \N$, and $\delta_{i,j}$ is the Kronecker delta function. We also show that 
The main result of the paper is a sufficient condition for log-concavity of triangular arrays which is the following.
\begin{theorem}\label{thm: main}
     For $n\ge 1$ and $0<k \le n$, let $T(n,k)$ satisfy the following triangular recurrence.
    \begin{equation*}
        T(n,k)=c(n,k)\ T(n-1,k)+d(n,k)\ T(n-1,k-1)
    \end{equation*}
    The sequence $(T(n,k))_{k=0}^n$ is log-concave if $(i),(ii)$ and $(iii)$ hold.
    \begin{enumerate}
        \item $c(n,k),d(n,k)$ are log-concave in $k$.
        \item  Suppose \begin{itemize}
     \item $n_2>n_1$; $l_2>k_2>k_1$; $l_2>l_1>k_1$,
     \item $k_2-k_1=l_2-l_1+1$,
     \item $n_2-n_1\ge l_2-k_2$ and $n_2-n_1\ge l_1-k_1$.
\end{itemize} Then,
$$c(n_1,k_1)d(n_1,k_2)d(n_2,l_1)c(n_2,l_2)\le c(n_1,k_1+1)d(n_1,k_2-1)d(n_2,l_1+1)c(n_2,l_2-1)$$ 
        \item For all $k$, we have $d(n_1,k+1)c(n_1,k)\le d(n_1,k)c(n_1,k+1)$.
    \end{enumerate}
 \end{theorem}

Using the theorem above, we obtain a sufficient condition for log-concavity of rows of arrays satisfying the Equation \eqref{Eq:super-recurrence}. 

\begin{theorem}\label{thm: a,b,c}
    Let $\alpha,\beta,\gamma$ and $\alpha',\beta',\gamma'$ be real numbers such that the following hold:
    \begin{enumerate}
        \item $\alpha, \alpha',\beta\ge 0$
        \item $\beta'\in [-\alpha',0]$
        \item $\alpha+\beta+\gamma,\ \alpha'+\beta'+\gamma'\ge 0$
    \end{enumerate}
    Then \begin{equation}
        T(n,k)=(\alpha n+\beta k+\gamma)^l\ T(n-1,k)+(\alpha'n+\beta'k+\gamma')^l \ T(n-1,k-1)
    \end{equation}
    is log-concave in $k$ for all $n$.
\end{theorem}

The main technique is to interpret the entries of the triangular array $(T(n,k))$ in terms of weighted lattice paths, which we describe in the next section.
Our method allows us to establish the log-concavity of the generalized Stirling and Lah numbers described above, as well as several other well-studied combinatorial sequences, which we list in the Table \ref{tab:parameters} below.

\begin{table}[ht]
\centering
\begin{tabular}{|l|l|l|}
\hline
\textbf{Parameters $(\alpha,\beta,\gamma;\  \alpha',\beta',\gamma')$} & \textbf{Name ($l=1$) and Reference} & \textbf{OEIS Entry} \\
\hline
$(0,1,1;\nu,-1,1-\nu)$ & $\nu$-order Eulerian numbers \cite{higher-order-eulerian} & A173018, A008517, A219512 \\
$(0,j,1;\ j,-j,0)$ & $1/j$-Eulerian numbers \cite{1/k-eulerian} & \\
$(0,1,0;\ 0,0,1)$ & Stirling subset numbers & A008277\\
$(1,1,-1;\ 0,0,1)$ & Lah numbers & A015278\\
$(1,m,-1;\ 0,0,m)$ & $m-$associated Lah numbers&\\
$(2,1,-2;\ 0,0,1)$ & Generalisation of Lah, Stirling numbers  \cite{stirling-lah} & A035342\\
$(r-1,1,1-r;\ 0,0,1)$ & $S(r,n,k)$ \cite{broder-r-stir}& \\
$(0,0,1;\ 0,0,1)$ & Binomial Coefficients & A007318\\
$(1,0,-1;\ 0,0,1)$ & Stirling cycle numbers & A132393\\
$(1,2,-1;\ 3,-2,1)$ & Scaled type A Narayana numbers \cite{ma-gamma-descent} & \\
$(1,2,0;\ 3,-2,0)$ & Scaled type B Narayana numbers \cite{ma-gamma-descent}& \\
$(2,1,-1;\ 0,0,1)$ & Holiday numbers of first kind \cite{Szekely-holiday}&\\
$(2,1,0;\ 0,0,1)$& Holiday numbers of second kind\cite{Szekely-holiday} & \\
\hline
\end{tabular}
\caption{Some sequences of combinatorial interest satisfying Theorem \ref{thm: a,b,c}.}
\label{tab:parameters}
\end{table}

 Motivated by the definitions of the $(l,r)$-Stirling and $(l,r)$-Lah numbers, we generalise the Eulerian numbers to a two parameter $(l,r)$-Eulerian number. 
A subexceedant function is a function $f:[n]\rightarrow [n]$ such that $0<f(i)\le i$ for all $i\in [n]$. Dumont (see \cite{mantaci-fanja-perm-rep}) showed that the Eulerian number $\eul(n,k)$ counts the number of subexceedant functions such that the cardinality of the image of $f$ is $k+1$. Similar to the Stirling and Lah counterparts, we define the block leader of a subexceedant function. The block leader of a subexceedant function $f$ is the set $\{i\in [n]: f(i)\notin f([n-1])\}$. Let us denote it by $\bl(f)$. Note that the $|\bl(f)|$ is the size of the image of $f$ because for each element in the image of $f$ the smallest element in the pre-image is in $\bl(f)$.

Let $l$ be a positive integer. Define $\eul_r^{(l)}(n,k)$ to count the number of $l$-tuples $(f_1,\dots, f_l)$ such that $[r]\subset \bl(f_1)$, $|\bl(f_1)|=k+1$, and $\bl(f_1)=\dots=\bl(f_l)$. We show that there is a natural connection 
between the $r$-Stirling numbers and $r$-Eulerian numbers (setting $l=1$), which is done by showing that the $r$-Eulerian
numbers count the number of permutations in $\SSS_n$ with $k$ descents such that $[r]$ is subset of the set of first letters 
of the increasing runs in $\pi$. Finally, we show that the associated polynomials of the $(l,1)$-Eulerian numbers are palindromic and make a remark about their gamma-nonnegativity and real-rootedness.

\section{Definitions and Preliminaries}
Suppose that $n\ge 1$ and $0 <k \le n$, a triangular recurrence is a recurrence of the form 
$$T(n,k)=c(n,k)\ T(n-1,k) + d(n,k)\ T(n-1, k-1) $$
with boundary conditions $T(1,1)=1$ and  and the coefficients $c(n,k), d(n,k)$ are non-negative real numbers. 

For $A, B \in \mathbb{Z}^2$, we denote by $\lp_{A \rightarrow B}$ the set of lattice paths starting with $A$ and ending at $B$ that use the step $N=(0,1)$ (north) and the step $C=(1,1)$ (cross). We can encode these paths as words in the alphabet $\{ N, C \}$. For example, we write $p=NCCCNN$ for the lattice path below.

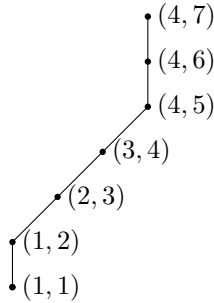
\begin{center}
    \begin{tikzpicture}[scale=0.6]

    \coordinate (A) at (1,1);
    \coordinate (B) at (1,2);
    \coordinate (C) at (2,3);
    \coordinate (D) at (3,4);
    \coordinate (E) at (4,5);
    \coordinate (F) at (4,6);
    \coordinate (G) at (4,7);

    \fill (1,1) circle (2pt) node[right] {$(1,1)$};
    \draw[-] (A) -- (B) node[right] {$(1,2)$};
    \draw[-] (B) -- (C) node[right] {$(2,3)$};
    \draw[-] (C) -- (D) node[right] {$(3,4)$};
    \draw[-] (D) -- (E) node[right] {$(4,5)$};
    \draw[-] (E) -- (F) node[right] {$(4,6)$};
    \draw[-] (F) -- (G) node[right] {$(4,7)$};

    \foreach \p in {A,B,C,D,E,F,G} {
        \fill[black] (\p) circle (2pt);
    }
\end{tikzpicture}
\end{center}

Since most of the paths in this paper will start from the point $(1,1)$, we will write $\lp_{A}$ instead of $\lp_{(1,1)\rightarrow A}$.

We shall assign weights to the step of the lattice path. For a north step $N$ that starts at $(k,n-1)$ and ends at $(k,n)$, we assign the weight $c(n,k)$. For the cross step $C$ that starts at $(k-1,n-1)$ and ends at $(k,n)$, we assign the weight $d(n,k)$. 

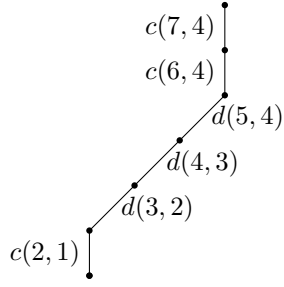
\begin{center}
    \begin{tikzpicture}[scale=0.6]

    \coordinate (A) at (1,1);
    \coordinate (B) at (1,2);
    \coordinate (C) at (2,3);
    \coordinate (D) at (3,4);
    \coordinate (E) at (4,5);
    \coordinate (F) at (4,6);
    \coordinate (G) at (4,7);

    \fill (1,1) circle (2pt);
    \draw[-] (A) -- (B) node[midway, left] {$c(2,1)$};
    \draw[-] (B) -- (C) node[midway, below, right] {$d(3,2)$};
    \draw[-] (C) -- (D) node[midway, below, right] {$d(4,3)$};
    \draw[-] (D) -- (E) node[midway, below, right] {$d(5,4)$};
    \draw[-] (E) -- (F) node[midway, left] {$c(6,4)$};
    \draw[-] (F) -- (G) node[midway, left] {$c(7,4)$};

    \foreach \p in {A,B,C,D,E,F,G} {
        \fill[black] (\p) circle (2pt);
    }
\end{tikzpicture}
\end{center}

The weight of a path $p \in \lp_{A \rightarrow B}$ is the product of the weights of the steps that make up the path. In our running example, it would be $c(2,1)d(3,2)d(4,3)d(5,4)c(6,4)c(7,4)$.
\section{A log-concavity criterion for triangular arrays}
\subsection{Triangular arrays}
We have the following interpretation of a triangular array in terms of weighted lattice paths:
\begin{theorem}
    For $n\ge 1$ and $0<k \le n$, let $T(n,k)$ satisfy the following triangular recurrence.
    \begin{equation*}
        T(n,k)=c(n,k)\ T(n-1,k)+d(n,k)\ T(n-1,k-1).
    \end{equation*}
Then,
    \begin{equation*}
        T(n,k)=\displaystyle \sum_{p \in \lp_{(k,n)}} \wt(p).
    \end{equation*}
\end{theorem}
\begin{proof}
    We prove by induction on $n$. Each path in $\lp_{(k,n)}$ is either path in $\lp_{(k,n-1)}$ followed by an $N$ step or a path in $\lp_{(k-1,n-1)}$ followed by a cross step $C$. Since the weight of a path is the product of weight of the constituting steps, the statement follows from the recurrence.
    \end{proof}

Under this interpretation, we have the following interpretation of the log-concavity of rows.
\begin{corollary}
    The rows of the triangular array $T(n,k)$ are log-concave if and only if
\begin{equation*}
    \displaystyle \sum_{(p,q)\in \lp_{(k+1,n)}\times \lp_{(k-1,n)}} \wt(p)\wt(q) \le \displaystyle \sum_{(p',q')\in \lp_{(k,n)}\times \lp_{(k,n)}} \wt(p')\wt(q')
\end{equation*}
\end{corollary}
\begin{proof}
    The inequality $T(n,k)^2\ge T(n,k-1)T(n,k+1)$ is equivalent to \begin{equation*}
    \bigg(\displaystyle \sum_{p \in \lp_{(k,n)}} \wt(p)\bigg)^2 \ge \bigg( \displaystyle \sum_{p\in \lp_{(k-1,n)}} \wt(p)\bigg) \bigg(\displaystyle \sum_{q\in \lp_{(k+1,n)}} \wt(q)\bigg)
\end{equation*}
This can be written as 
\begin{equation*}
    \displaystyle \sum_{(p',q')\in \lp_{(k,n)}\times \lp_{(k,n)}} \wt(p')\wt(q') \ge \displaystyle \sum_{(p,q)\in \lp_{(k+1,n)}\times \lp_{(k-1,n)}} \wt(p)\wt(q). 
\end{equation*}
This completes our proof.
\end{proof}
\subsection{An injection from $\lp_{(k+1,n)}\times \lp_{(k-1,n)}$ to $\lp_{(k,n)}\times \lp_{(k,n)}$}
We shall describe an injective map that sends a pair of paths $(p,q)\in \lp_{(k+1,n)}\times \lp_{(k-1,n)}$ to a pair of paths $(p',q')\in \lp_{(k,n)}\times \lp_{(k,n)}$. The injection is based on an idea of Sagan \cite{sagan-log-concave} that is used to give injective proofs of log-concavity of the Stirling numbers of first and second kind. 

\begin{theorem}[Discrete Intermediate Value Theorem]\label{prop: DIVT}
Let $f$ be an integer-valued function on the integers in the interval $[m,n]$. Suppose that $|f(i+1)-f(i)|\le 1$ for $m\le i\le n$. If $f(m)f(n)<0$, then $f(x)=0$ for some integer $x \in [m,n]$.    
\end{theorem}

\begin{proposition}
    Let $(p,q)\in \lp_{(k+1,n)}\times \lp_{(k-1,n)}$ with $p=p_1\dots p_{n-1}$ and $q=q_1\dots q_{n-1}$. There is an index $1\le i\le n-1$ such that $p'=p_1\dots p_i q_{i+1}\dots q_{n-1}$ and $q'=q_1\dots q_i p_{i+1}\dots p_{n-1}$ are in $\lp_{(k,n)}\times \lp_{(k,n)}$.
\end{proposition}
\begin{proof}
Let $N:[0,n-1] \rightarrow \mathbb Z$ be the function defined by $N(0)=-1$ and $N(i)=\#\{ r\in [i]\ |\  p_r=C\} - \#\{ r \in [i]\ |\  q_r=C \}-1$ for $1\le i\le n-1$. Here, $|N(i+1)-N(i)|=1$ if either $p_{i+1}=C$ but $q_{i+1}\neq C$, or $p_{i+1}\neq C$ but $q_{i+1}=C$, and $|N(i+1)-N(i)|=0$ otherwise. Also, $N(0)=-1$ and $N(n-1)=1$ making $N(0)N(n-1)=-1<0$. Therefore, by Theorem \ref{prop: DIVT}, we have that $N(x)=0$ for some $x\in [n-1]$. Since $N(n-1)-N(x)=1$, we have that there is one more $C$ step in the remaining subword of $p$ than there is for $q$. Therefore, switching the subwords will have the intended effect.
\end{proof}
We are ready to define an injection $I$ from $\lp_{(k+1,n)}\times \lp_{(k-1,n)}$ to $\lp_{(k,n)}\times \lp_{(k,n)}$.
Given $(p,q) \in \lp_{(k+1,n)}\times \lp_{(k-1,n)}$ with $p=p_1\dots p_{n-1}$ and $q=q_1\dots q_{n-1}$, choose the largest index $i$ such that $p'=p_1\dots p_i q_{i+1}\dots q_{n-1}$ and $q'=q_1\dots q_i p_{i+1}\dots p_{n-1}$ belongs to $\lp_{(k,n)}\times \lp_{(k,n)}$. Then $I(p,q):=(p',q')$. The map is well-defined because of the previous proposition.

\begin{remark}\label{rem: imp}
    A step in path $q\in \lp_{k-1,n}$ after the index $i$ is moved exactly $1$ step to the right and a step that started at $(k_1,n_1)$ would now start at $(k_1+1,n_1)$ in the path $q'\in \lp_{(k,n)}$. Similarly, any step in path $p\in \lp_{k+1,n}$ would move exactly $1$ step to the left after the  index $i$. Any step that started at $(k_2,n_1)$ would now start at $(k_2-1,n_1)$ in the path $p'\in \lp_{(k,n)}$.
\end{remark}

\begin{example}
    Take the two paths $NNCNNN\in \lp_{(2,7)}$ and $NCCCNN\in \lp_{(4,7)}$. Under the injection, it would go to $NNCCNN$ and $NCCNNN$ in $\lp_{(3,7)}$. The index $i$ where the shift occurs is $3$.
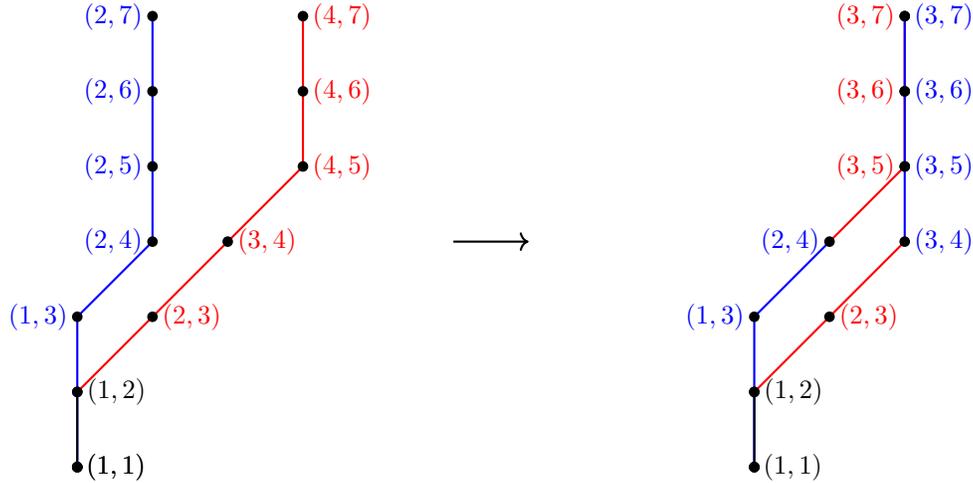
\begin{center}
\begin{tikzpicture}[scale=1]
    \coordinate (A) at (1,1);
    \coordinate (B) at (1,2);
    \coordinate (C) at (1,3);
    \coordinate (D) at (2,4);
    \coordinate (E) at (2,5);
    \coordinate (F) at (2,6);
    \coordinate (G) at (2,7);
    
    \draw[-, thick, blue] (1,1) -- (1,2) ;
    \draw[-, thick, blue] (1,2) -- (1,3) node[left] {$(1,3)$};
    \draw[-, thick, blue] (1,3) -- (2,4) node[left] {$(2,4)$};
    \draw[-, thick, blue] (2,4) -- (2,5) node[left] {$(2,5)$};
    \draw[-, thick, blue] (2,5) -- (2,6) node[left] {$(2,6)$};
    \draw[-, thick, blue] (2,6) -- (2,7) node[left] {$(2,7)$};

    \foreach \p in {A,B,C,D,E,F,G} {
        \fill[black] (\p) circle (2pt);
    }

    \coordinate (A) at (1,1);
    \coordinate (U) at (1,2);
    \coordinate (V) at (2,3);
    \coordinate (W) at (3,4);
    \coordinate (X) at (4,5);
    \coordinate (Y) at (4,6);
    \coordinate (Z) at (4,7);

    \fill (1,1) circle (2pt) node[right] {$(1,1)$};
    \draw[-, thick] (A) -- (U) node[right,black] {$(1,2)$};
    \draw[-, thick, red] (U) -- (V) node[right] {$(2,3)$};
    \draw[-, thick, red] (V) -- (W) node[right] {$(3,4)$};
    \draw[-, thick, red] (W) -- (X) node[right] {$(4,5)$};
    \draw[-, thick, red] (X) -- (Y) node[right] {$(4,6)$};
    \draw[-, thick, red] (Y) -- (Z) node[right] {$(4,7)$};

    \foreach \p in {A,U,V,W,X,Y,Z} {
        \fill[black] (\p) circle (2pt);
    }

   \draw[->,thick] (6,4)--(7,4);
    
\coordinate (a) at (10,1);
    \coordinate (b) at (10,2);
    \coordinate (c) at (10,3);
    \coordinate (d) at (11,4);
    \coordinate (e) at (12,5);
    \coordinate (f) at (12,6);
    \coordinate (g) at (12,7);
    
    \fill (10,1) circle (2pt) node[right] {$(1,1)$};
    \draw[-, thick, blue] (10,1) -- (10,2);
    \draw[-, thick, blue] (10,2) -- (10,3) node[left] {$(1,3)$};
    \draw[-, thick, blue] (10,3) -- (11,4) node[left] {$(2,4)$};
    \draw[-, thick, red] (11,4) -- (12,5) node[left] {$(3,5)$};
    \draw[-, thick, red] (12,5) -- (12,6) node[left] {$(3,6)$};
    \draw[-, thick, red] (12,6) -- (12,7) node[left] {$(3,7)$};

    \foreach \p in {a,b,c,d,e,f,g} {
        \fill[black] (\p) circle (2pt);
    }
\coordinate (a) at (10,1);
    \coordinate (u) at (10,2);
    \coordinate (v) at (11,3);
    \coordinate (w) at (12,4);
    \coordinate (x) at (12,5);
    \coordinate (y) at (12,6);
    \coordinate (z) at (12,7);

    \fill (1,1) circle (2pt) node[right] {$(1,1)$};
    \draw[-, thick] (a) -- (u) node[right,black] {$(1,2)$};
    \draw[-, thick, red] (u) -- (v) node[right] {$(2,3)$};
    \draw[-, thick, red] (v) -- (w) node[right,blue] {$(3,4)$};
    \draw[-, thick, blue] (w) -- (x) node[right,blue] {$(3,5)$};
    \draw[-, thick, blue] (x) -- (y) node[right,blue] {$(3,6)$};
    \draw[-, thick, blue] (y) -- (z) node[right,blue] {$(3,7)$};

    \foreach \p in {a,u,v,w,x,y,z} {
        \fill[black] (\p) circle (2pt);
    }
\end{tikzpicture}
\end{center}
\end{example}

In the next section, we give sufficient conditions for ensuring that $\wt(p,q) \le \wt I(p,q)$ for $(p,q)\in \lp_{(k+1,n)}\times \lp_{(k-1,n)}$.

\subsection{Proofs of Theorems \ref{thm: main} and \ref{thm: a,b,c}}

Usually, a $2$-Motzkin path of length $n$ is a lattice path starting from $(0,0)$ and $(n,0)$ that uses an up step $U:=(1,1)$, a down step $D:=(1,-1)$, a blue horizontal step $H_1:=(1,0)$ and a red horizontal $H_2:=(1,0)$. However, we will use the name $2$-Motzkin path to mean a lattice path with the steps described above but starting from $(0,0)$ and ending in $(n,-1)$ such that the last step is a down step $D$ and except for the last step, the remaining lattice path lies entirely above the $x$-axis.

To obtain a sufficient condition for log-concavity, we will associate a weighted $2$-Motzkin path to the pair of paths in $\lp_{(k+1,n)}\times \lp_{(k-1,n)}$ and use the weighted $2$-Motzkin path to create a sufficient condition that ensures the weight of pair of paths increases under the injection defined in the previous subsection. Let $(p,q)\in \lp_{(k+1,n)}\times \lp_{(k-1,n)}$. Let $\ch(p):=p_{i+1}\dots p_{n-1}$ and $\ch(q):=q_{i+1}\dots q_{n-1}$ where $i$ is largest index such that $p_1\dots p_i \ch(q)$ and $q_1\dots q_i \ch(p)$ are in $\lp_{(k,n)}$. 

The idea will be that the pair of paths $p,q$, when looked at from the end, start out with a spacing of $2$ ($(k-1,n)$ and $(k+1,n)$), continue with a spacing greater than or equal to $2$ and finally, getting down to a spacing of $1$ at the index $i$ prescribed by the injection. We will encode this portion of the paths as a $2$-Motzkin path.

To this end, consider the reversal operator $\rev$ on words. We start with the two words $\rev(\ch(p))$ and $\rev(\ch(q))$. We build a weighted $2$-Motzkin path out of the two words as follows. If the $j^{th}$ steps of $\rev(\ch(p))$ and $\rev(\ch(q))$ are $N$ and $C$ respectively (resp. $C$ and $N$ or $N$ and $N$ or $C$ and $C$), the $j^{th}$ step of our weighted $2$-Motzkin path is $U$ (resp. $D,H_1,H_2$). The weight of the step in the $2$-Motzkin path will be the product of the weight of the corresponding steps in $p$ and $q$. 

The fact that such a construction yields a weighted $2$-Motzkin path follows the fact that $i$ is the largest index such that $\ch(p)$ has one more $C$ step than $\ch(q)$. The $2$-Motzkin path will start at $(0,0)$ and end at $(n-i,-1)$. Except for the last $D$ step, the entire path will be above the $x$-axis. 

The $2$-Motzkin path corresponding to our pair of paths is below. The blue steps correspond to $H_1$ steps.
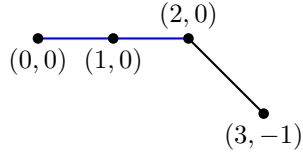
\begin{center}
    \begin{tikzpicture}[scale=1]
    \coordinate (A) at (0,0);
    \coordinate (B) at (1,0);
    \coordinate (C) at (2,0);
    \coordinate (D) at (3,-1);
 \draw[-,thick, blue] node[below, black] {$(0,0)$} (A)--(B) node[below,black] {$(1,0)$} ;
\draw[-,thick, blue]  (B)--(C) node[above, black] {$(2,0)$} ;
\draw[-,thick] (C)--(D)  node[below,black] {$(3,-1)$};

\foreach \p in {A,B,C,D} {
        \fill[black] (\p) circle (2pt);
    }
\end{tikzpicture}
\end{center}

Let us calculate the weight of the four different steps. The weight of $U$ step is $c(n_1,k_1)d(n_1,k_2)$, and the $D$ step is $d(n_1,k_1)c(n_1,k_2)$. The weight of a $H_1$ step is $c(n_1,k_1)c(n_1,k_2)$. The weight of a $H_2$ step is $d(n_1,k_1)d(n_1,k_2)$.

 On a $2$-Motzkin path, for each $U$ step, there is a corresponding $D$ step. That is if $(n_2,l_2)\in p$ and $(n_2,l_1)\in q$, there is a $(n_1,k_2)\in p$ and $(n_1,k_1)\in q$ such that the following are true.
 \begin{itemize}
     \item $n_2>n_1$, $l_2>k_2>k_1$, $l_2>l_1>k_1$,
     \item $k_2-k_1=l_2-l_1+1$,
     \item $n_2-n_1\ge l_2-k_2$ and $n_2-n_1\ge l_1-k_1$.
 \end{itemize}

We are in a position to prove Theorem \ref{thm: main}.
\begin{proof}[Proof of Theorem \ref{thm: main}]
The weight of a $H_1$ step is $c(n_1,k_1)c(n_1,k_2)$ which goes, under the injection, to a $H_1$ step whose weight is $c(n_1,k_1+1)c(n_1,k_2-1)$ (see Remark \ref{rem: imp}). Similarly, for $H_2$, $d(n_1,k_1)d(n_1,k_2)$ goes to $d(n_1,k_1+1)d(n_1,k_2-1)$. Therefore, Condition $(i)$ ensures that $\wt(H_1)\le \wt(I(H_1))$ and $\wt(H_2)\le \wt(I(H_2))$.

Similarly, Condition $(ii)$ ensures that the product of the weights of a $U$ and its corresponding $D$ step in the path is greater after applying the injection. 

Finally, Condition $(iii)$ ensures that the final $D$ step is greater after applying the injection.

These three conditions ensure that 
\begin{equation*}
    \displaystyle \sum_{(p,q)\in \lp_{(k+1,n)}\times \lp_{(k-1,n)}} \wt(p)\wt(q) \le \displaystyle \sum_{(p',q')\in \lp_{(k,n)}\times \lp_{(k,n)}} \wt(p')\wt(q')
\end{equation*}
Therefore, the sequence $T(n,k)$ is log-concave in $k$.
\end{proof}

\begin{remark}
    It should be possible to weaken Condition $(iii)$ or completely remove it by sufficiently strengthening the other two conditions.
\end{remark}

The following is a quick application of this theorem. The Legendre-Stirling numbers were shown to be log-concave \cite{andrews-legendre-stirling}. It is easy to verify that the coefficients satisfy the above theorem, therefore, giving us another proof of their log-concavity.
\begin{corollary}
    The Legendre-Stirling numbers defined by
    \begin{equation}
        LS(n,k)=(k^2+k)LS(n-1,k)+LS(n-1,k-1)
    \end{equation}
    are log-concave in $k$ for all $n$.
\end{corollary}

The following is an observation that follows from the form of the conditions in Theorem \ref{thm: main}.
\begin{corollary}\label{cor: multiplicative}
    For $c(n,k),d(n,k)\ge 0$, if the triangular array $$T(n,k)=c(n,k)T(n-1,k)+d(n,k)T(n-1,k-1)$$ satisfies the conditions of Theorem \ref{thm: main}, then for $l\in \mathbb{N}$, so does the array $$\widetilde{T}(n,k)=c(n,k)^l\widetilde{T}(n-1,k)+d(n,k)^l\widetilde{T}(n-1,k-1).$$ 
\end{corollary}
Using the two aforementioned results, we can prove the main result of the paper.
\begin{proof}[Proof of Theorem \ref{thm: a,b,c}]
    Condition (i) is easily verified.
    We insist that $\alpha, \alpha'\ge 0$. Otherwise, for all large $n$, we would have $\alpha n +\beta k+\gamma$ or $\alpha'n+\beta' k+\gamma'$ becoming negative. Similarly, we insist on $\alpha+\beta+\gamma, \alpha'+\beta'+\gamma'\ge 0$ for ensuring the non-negativity of $c(n,k),d(n,k)$. We want to show that conditions $(ii), (iii)$ of Theorem \ref{thm: a,b,c} are enough to apply Theorem \ref{thm: main}. 
    
    $\textbf{Verifying Condition (ii):}$ 
    \begin{eqnarray*}
         &&c(n_1,k_1+1)c(n_2,l_2-1)-c(n_1,k_1)c(n_2,l_2)\\
         &=& (\alpha n_1 +\beta k_1+\gamma+\beta)(\alpha n_2+\beta l_2+\gamma-\beta)-(\alpha n_2+\beta k_1+\gamma)(\alpha n_2+\beta l_2+\gamma)\\
         &=& \alpha\beta (n_2-n_1)+\beta^2(l_2-k_1-1) \ge 0
\\
\\
       &&d(n_1,k_2-1)d(n_2,l_1+1)-d(n_1,k_2)d(n_2,l_1)\\
        &=&(\alpha' n_1+\beta' k_2+\gamma'-\beta')(\alpha'n_2+\beta'l_1+\gamma'+\beta')-(\alpha'n_1+\beta'k_2+\gamma')(\alpha'n_2+\beta'l_1+\gamma')\\
        &=& -\beta' ( \alpha'(n_2-n_1)+ \beta'(l_1-k_2+1)) \ge 0
    \end{eqnarray*}
    The last line follows because $\alpha'(n_2-n_1)\ge\alpha'(l_1-k_1)\ge\alpha'(l_1-k_2+1)\ge|\beta'|(l_1-k_2+1)$ and so, $ \alpha'(n_2-n_1)+ \beta'(l_1-k_2+1)\ge 0$.

    $\textbf{Verifying Condition (iii):}$  
    \begin{eqnarray*}
     &&\frac{d(n_1,k+1)}{d(n_1,k)}\frac{c(n_1,k)}{c(n_1,k+1)}\le 1\\
    \end{eqnarray*}
    Since $\beta'\le 0$, $d(n,k)$ increases as $k$ decreases and since $\beta\ge0$, $c(n,k)$ increases as $k$ increases.
\end{proof}
\begin{remark}
    The conditions $\alpha+\beta+\gamma\ge 0$ and $\alpha'+\beta'+\gamma'\ge 0$ exist to ensure that coefficients $c,d$ are non-negative. One may increase $\gamma,\gamma'$ by non-negative integers and retain log-concavity of the sequence.
\end{remark}

\begin{corollary}
    The $(l,r)$-Lah numbers are log-concave in $k$ for all $n$.
\end{corollary}
\begin{proof}
    The $(l,r)$-Lah numbers satisfy the recurrence relation
    \begin{equation}
        T(n,k)=(n+k+2r-3)^l\ T(n-1,k)+T(n-1,k-1)
    \end{equation}
    with the initial condition $T(0,0)=1$. By Theorem \ref{thm: a,b,c}, we have that this triangular array is log-concave in $k$ for all $n$.
\end{proof}
\section{Two parameter generalization of Eulerian numbers}\label{sec: 2-para }

Recall that we defined $\eul_r^{(l)}(n,k)$ to count the number of $l$-tuples $(f_1,\dots, f_l)$ such that $[r]\subset \bl(f_1)$, $|\bl(f_1)|=k+1$, and $\bl(f_1)=\dots=\bl(f_l)$. We first start out by proving a recurrence relation for the $(l,r)$-Eulerian numbers. They satisfy a triangular recurrence analogous to the ones satisfied by their Stirling and Lah counterparts.
\begin{proposition}
    For $n\ge k\ge r-1\ge0$ and a positive integer $l$, we have 
    \begin{equation}
        \eul_r^{(l)}(n,k)=(n-k)^l \eul_r^{(l)}(n-1,k-1)+(k+1)^l\eul_r^{(l)}(n-1,k)
    \end{equation}
    and $\eul_r^{(l)}(r,r)=1$
\end{proposition}
\begin{proof}
     Let $f$ be a subexceedant function with $k+1$ elements in $\bl(f)$ and $[r]\subset \bl(f)$. Then, $f$ could have been obtained in one of the following two ways.
    \begin{enumerate}
        \item if $n\in \bl(f_i)$, its image has one of $(n-k)^l$ choices,
        \item if not, then its image has one of $(k+1)^l$ choices.
    \end{enumerate}
\end{proof}

\begin{remark}
      By Theorem \ref{thm: a,b,c}, the $(l,r)$-Eulerian numbers are log-concave and therefore, unimodal.    
\end{remark}

\subsection{A bijection}
Let $\pi\in \SSS_n$. 
Define the block leaders of $\pi$, denoted by $\bl(\pi)$, as the set of first letters of the increasing runs of $\pi$. For example, let $\pi=28|7|4|136|5$ with bars separating increasing runs. Then, $\bl(\pi)=\{ 1,2,4,5,7 \}$. 
We show that $\eul_r(n,k):=\eul^{(1)}_r(n,k)$ also counts the number of permutations $\pi\in \SSS_n$ with $k$ descents such that $[r]\subset \bl(\pi)$.

Starting from a subexceedant function $f$ that satisfies $[r]\subset \bl(f)$ and $|\bl(f)|=k+1$, we give a map $\Lambda$ to a permutation $\pi$ with $k$ descents such that $[r]\subset \bl(\pi)$ with the additional property that $f(\bl(f))=\bl(\pi)$.

We give the bijection $\Lambda$ inductively. Start with the empty word and $i=1$. If $f(i)=i$, add $i$ to the left of all existing letters. If $f(i)\in f([i-1])$, then add $i$ at the end of the increasing run starting at $f(i)$. If $f(i)\notin f([i-1])$, then add $i$ immediately to the left of element $f(i)$. Increase $i$ to $i+1$ and repeat.

\begin{example}
    Starting with the subexceedant function $f=12121547$, we show the steps.
    \begin{eqnarray*}
        \phi\rightarrow 1 \rightarrow 21\rightarrow 213\rightarrow 2413\rightarrow 24135\\
        \rightarrow 241365 \rightarrow 2741365\rightarrow 28741365
    \end{eqnarray*}
    Therefore, $\Lambda(12121547)=28741365$.
    Note that the $f(\bl(f))=\{ 1,2,4,5,7 \} =\bl(28741365)$.
\end{example}

We can reconstruct the subexceedant function given a permutation $\pi$ as follows. For $1\le i\le n$, let $\pi|_i$ be the permutation with all letters greater than $i$ removed. In our running example, $\pi|_5=24135$.
If $i$ appears to the left of all elements in $\pi|_{i}$, then $f(i)=i$. If $i$ appears in front of some element $j\notin f([i-1])$, then $f(i)=j$. Else, $f(i)$ is the closest element from $f([i-1])$ to the left of $i$. 
\begin{example}
    Start with $\pi=28741365$. Here, $\pi|_1=1$, $\pi|_2=21$, $\pi|_3=213$, $\pi|_4=2413$, $\pi|_5=24135$, $\pi|_6=241365$, $\pi|_{7}=2741365$ and $\pi|_8=28741365$. 
    It is easy to check that $f(1)=1,f(2)=2$. $f(3)=1$ because $1$ the closest element in $f([2])$ to the left of $3$. $f(4)=2$, $f(5)=1$, in the same way. $f(6)=5$ because $5\notin f([5])$. Similarly, $f(7)=4$ and $f(8)=7$.
\end{example}

\begin{proposition}
    The map $\Lambda$ from the set of subexceedant functions from $[n]$ to $[n]$ such that $[r]\in \bl(f)$, $|\bl(f)|=k+1$ to the set of permutations on $[n]$ such that $[r]\in \bl(\pi)$, $|\bl(\pi)|=k+1$ is a bijection such that $f(\bl(f))=\bl(\Lambda(f))$.
\end{proposition}
\begin{proof}
    It is clear that the function $\Lambda$ is a bijection as we are able to describe its inverse. It remains to show that $f(\bl(f))=\bl(\Lambda(f))$. Suppose $f(i)\notin f([i-1])$, then $i$ goes behind $f(i)$, making $f(i)$ the first element of an increasing run. Similarly, the elements $i$ such that $f(i)=i$ are added to the leftmost position and therefore, are first elements of an increasing run. Therefore, the block leaders of $\Lambda(f)$ are exactly the elements of $f(\bl(f))$.
\end{proof}
We have the following connection between the $r$-Stirling number of the second kind and $\eul_r(n,k)$. When $r=1$, this identity specialises to B\'ona \cite[Theorem 1.18]{bona-combin-permut}.
\begin{theorem} For $n\ge k\ge r-1\ge 0$, we have
    \begin{equation}\label{r-stir-eul}
        k!S_r(n,k)=\sum_{i=r}^{k}\eul_r(n,i-1)\binom{n-i}{k-i}.
    \end{equation}
\end{theorem}
\begin{proof}
    The LHS is the number of ordered partitions $\Pi$ of $[n]$ into $k$ blocks such that $[r] \subset \bl(\Pi)$. We need to show that RHS counts the same objects. Let $\pi\in \SSS_n$ with $i-1$ descents and $[r]\subset \bl(\pi)$. The $i$ increasing runs of $\pi$ give us an ordered partition of $[n]$. Now, the number of blocks is $i$, which can be smaller than $k$. However, we can make more blocks by making splits of the increasing runs. We need to make $k-i$ splits. This can be done by choosing $k-i$ of the $n-i$ different places where we can split.

    We just need to now show that the ordered partitions thus formed are all distinct. However, given an ordered partition $\Pi$ of $[n]$, we can arrange the blocks in ascending order internally and read from left to right a unique permutation and the choice of splitting from which it can be obtained.
\end{proof}

\begin{corollary} For $n\ge k\ge r-1\ge 0$, we have
    \begin{equation}
        \eul_r(n,k-1)=\sum_{i=1}^{k}(-1)^{k-i}S_r(n,i)\binom{n-i}{k-i}i!
    \end{equation}
\end{corollary}
\begin{proof}
    Equation \ref{r-stir-eul} can be rewritten as $\sum_{k=r}^n\eul_r(n,k-1)(t+1)^k=\sum_{k=r}^nk!S_r(n,k)t^{n-k}$.
    The corollary follows from the substitution $t\rightarrow t-1$.
\end{proof}
\subsection{Palindromicity of the associated generalised Eulerian polynomials}
In this section, let us restrict our attention to the case when $r=1$.
We define associated $(l,1)$-Eulerian polynomials to be
\begin{equation}
    \eul_n^{(l)}(t):=\sum_{k=0}^{n-1} \eul^{(l)}(n,k)t^k.
\end{equation}
It is not immediately clear that this polynomial is palindromic. To that end, we define \begin{equation}\label{eq: r=1 eul}
     \eul_n^{(l)}(s,t):= \displaystyle \sum_{k=0}^{n-1} \eul^{(l)}(n,k)s^kt^{n-1-k}.
\end{equation} 
We show that the polynomial $\eul_n^{(l)}(s,t)$ is symmetric in $s,t$, i.e.,
$\eul_n^{(l)}(s,t)=\eul_n^{(l)}(s,t)$. Let $\dd_s=s\frac{\partial }{\partial s}$ and $\dd_t=t\frac{\partial }{\partial t}$ in $\mathbb Q[s,t]$.
The symmetry becomes apparent from the following proposition.

\begin{proposition}\label{prop: diff-recurrence}
    The polynomials $\eul^{(l)}_n(s,t)$ satisfy the following partial differential equation.
    \begin{equation}
        \eul^{(l)}_n(s,t)= \frac{1}{s} \dd_s^l(st\eul^{(l)}_{n-1}(s,t))+ \frac{1}{t} \dd_t^l(st\eul^{(l)}_{n-1}(s,t))
    \end{equation}
    with $\eul^{(l)}_0(s,t)=1$ initially.
\end{proposition}
\begin{proof}
    \begin{eqnarray*}
        \dd_s^l(st\eul^{(l)}_{n}(s,t)) & = & t \dd_s^l(s\eul^{(l)}_{n}(s,t))\\
        & = & t\bigg( \dd_s^l(\sum_{k=0}^{n} \eul^{(l)}(n,k) s^{k+1} t^{n-1-k}) \bigg)\\ 
        & = & \sum_{k=0}^{n-1}(k+1)^l\eul^{(l)}(n,k) s^{k+1} t^{n-k}
    \end{eqnarray*}
    Similarly,
    \begin{eqnarray*}
        \dd_t^l(st\eul^{(l)}_{n}(s,t)) & = & s \dd_t^l(t\eul^{(l)}_{n}(s,t))\\
        & = & s\bigg( \dd_t^l(\sum_{k=0}^{n} \eul^{(l)}(n,k) s^{k} t^{n-k}) \bigg)\\ 
        & = & \sum_{k=0}^{n-1}(n-k)^l\eul^{(l)}(n,k) s^{k+1} t^{n-k}
    \end{eqnarray*}
    Therefore, the RHS of the proposition evaluates the polynomial
    \begin{equation*}
        \sum_{k=0}^{n-1}(k+1)^l\eul^{(l)}(n,k) s^{k} t^{n-k}+\sum_{k=0}^{n-1}(n-k)^l\eul^{(l)}(n,k) s^{k+1} t^{n-1-k}
    \end{equation*}
    Comparing coefficients with \eqref{eq: r=1 eul} gives us the result. 
\end{proof}

A palindromic polynomial $p(x)=\displaystyle\sum_{k=0}^n a_nx^k$ is said to be gamma-positive of center of symmetry $\lfloor \frac{n}{2}\rfloor$ if it can be written in the form $p(x)=\displaystyle\sum_{i=0}^{\lfloor\frac{n}{2} \rfloor} \gamma_{n,i}t^i(1+t)^{n-2i}$.
In light of the palindromicity of these polynomials, it would be the next natural question to ask if they are gamma-positive. It is not
difficult to show that these polynomials are gamma-positive when $l=2$.
The following is the recurrence \ref{prop: diff-recurrence} for the polynomials written out. 
\begin{proposition}
    The polynomials $\eul_n^{(2)}(s,t)$ satisfy the recurrence relation
    \begin{equation}
        \eul_n^{(2)}(s,t) = (s+t)\eul_{n-1}^{(2)}(s,t) + 3st(\frac{\partial }{\partial s}
        +\frac{\partial }{\partial t})\eul_{n-1}^{(2)}(s,t)+ st(s\frac{\partial^2 }{\partial s^2}
         + t\frac{\partial^2 }{\partial t^2})\eul_{n-1}^{(2)}(s,t)
    \end{equation}
    with $\eul_0^{(2)}(s,t)=1$ initially.
\end{proposition}
It is straightforward to show that $$s\frac{\partial^2 }{\partial s^2}
+ t\frac{\partial^2 }{\partial t^2}$$
sends an element of the gamma-basis $(st)^i(s+t)^{n-1-i}$ to a positive linear combination of gamma basis elements as do 
the other two terms. Since these are linear operators, we have that the polynomials are gamma-positive by induction. If we set $s=1$, we get

\begin{proposition}
    The polynomials $\eul_n^2(t):=\displaystyle \sum_{k=0}^{n-1} \eul^{(2)}(n,k)t^k$ are gamma-positive for all $n\ge 2$.
\end{proposition}

Here are the polynomials $\eul_n^3(t)$ and their expansion in the gamma-basis.
\begin{eqnarray*}
    \eul_3^3(t)&=& 1+16t+t^2=(1+t)^2+14t\\
    \eul_4^3(t)&=& 1+155t+155t^2+t^3=(1+t)^3+152t(1+t)\\
    \eul_5^3(t)&=& 1+1304t+8370t^2+1304t^3+t^4=(1+t)^4+1300t(1+t)^2+5764t^2\\
    \eul_6^3(t)&=& 1+10557t+309446t^2+309446t^3+10557t^4+t^5=(1+t)^5+10552t(1+t)^3+277780t^2(1+t)
\end{eqnarray*}

\begin{remark}
We know that the Eulerian polynomials are gamma-positive and we have proved that the $(2,1)$-Eulerian polynomials are gamma-positive.
    Based on the limited data that we were able to compute, the $(l,1)$-Eulerian polynomials seem to be gamma-positive for all $l\ge 3$.
\end{remark}

\begin{remark}
    A similar remark can be made about real-rootedness. The Eulerian polynomials are known to be real rooted due to Frobenius \cite{frobenius-real-root-eulerian}. From our limited data, it appears that the $(l,1)$-Eulerian polynomials are real rooted for $l\ge 2$. 
\end{remark}

\bibliographystyle{acm}
\end{document}